\documentclass[a4paper]{article}
\usepackage{amsmath, amsthm, amssymb, enumerate}
\newcommand{\abs}[1]{\lvert#1\rvert}
\newtheorem{thm}{Theorem}
\newtheorem{lem}[thm]{Lemma}
\theoremstyle{remark}

\newtheorem{clm}{Claim}
\newtheorem{poc}{Proof of Claim}
\newcommand{\thrd}{\cdots}
\newcommand{\mmm}{\operatorname{mmm}}

\makeatletter
\let\@fnsymbol\@arabic
\makeatother

\newcommand\thankx[1]{\begingroup\let\rlap\relax\thanks{#1}\endgroup}

\begin{document}

\title{Subdivisions in the Robber Locating game}
\author{John Haslegrave\thankx{University of Sheffield, Sheffield, UK. {\tt j.haslegrave@cantab.net}}, Richard A. B. Johnson\thankx{University of Memphis, Memphis TN, USA. {\tt rabjohnson@gmail.com}}, Sebastian Koch\thankx{University of Cambridge, Cambridge, UK. {\tt sk629@cam.ac.uk}}}
\maketitle

\begin{abstract}
We consider a game in which a cop searches for a moving robber on a graph using distance probes, which is a slight variation on one introduced by Seager. Carragher, Choi, Delcourt, Erickson and West showed that for any $n$-vertex graph $G$ there is a winning strategy for the cop on the graph $G^{1/m}$ obtained by replacing each edge of $G$ by a path of length $m$, if $m\geqslant n$. They conjectured that this bound was best possible for complete graphs, but the present authors showed that in fact the cop wins on $K_n^{1/m}$ if and only if $m\geqslant n/2$, for all but a few small values of $n$. In this paper we extend this result to general graphs by proving that the cop has a winning strategy on $G^{1/m}$ provided $m\geqslant n/2$ for all but a few small values of $n$; this bound is best possible. We also consider replacing the edges of $G$ with paths of varying lengths.
\end{abstract}

\section{Introduction}

Pursuit and evasion games on graphs have been widely studied, beginning with the introduction by Parsons \cite{Par76} of a game where a fixed number of searchers try to find a lost spelunker in a dark cave. The searchers cannot tell where the target is, and aim to move around the vertices and edges of the graph in such a way that one of them must eventually encounter him. The spelunker may move around the graph in an arbitrary fashion, and in the worst case may be regarded as an antagonist who knows the searchers' positions and is trying to escape them.

The best-known variant is the classical Cops and Robbers game, introduced by Quillot \cite{Qui78}, and independently in a paper of Nowakowski and Winkler \cite{NW83} (where it is attributed to G. Gabor). Unlike the Lost Spelunker game, Cops and Robbers is played with perfect information, so that at any time each of the agents knows the location of all others. A fixed number of cops take up positions on vertices of a connected graph and a robber then starts on any unoccupied vertex. The cops and the robber take turns: at his turn the robber may move to any adjacent vertex or remain where he is, and at their turn all cops simultaneously make moves of this form. The cops win if at any point one of them reaches the robber's location. On a particular graph $G$ the question is whether a given number of cops have a strategy which is guaranteed to win, or whether there is a strategy for the robber which will allow him to evade capture indefinitely. The \textit{cop number} of a graph is the minimum number of cops that can guarantee to catch the robber.

Early results on this game include those obtained by Nowakowski and Winkler \cite{NW83}, who categorised the graphs of cop number 1, and Aigner and Fromme \cite{AF84}, who showed that every planar graph has cop number at most 3. An important open problem is Meyniel's conjecture, published by Frankl \cite{Fra87}, that the cop number of any $n$-vertex connected graph is at most $O(\sqrt{n})$ -- this has been shown to be true up to a $\log(n)$ factor for random graphs by Bollob{\'a}s, Kun and Leader \cite{BLK13}, following which \L uczak and Pra\l at improved the error term \cite{LP10}. More recently several variations on the game have been analysed by Clarke and Nowakowski (e.g. \cite{CN00}).

In this paper we consider the Robber Locating game, introduced in a slightly different form by Seager \cite{Sea12} and further studied by Carragher, Choi, Delcourt, Erickson and West \cite{CCDEW}. Like the Lost Spelunker game, the focus is on locating a hidden target, but, like Cops and Robbers, the target is a robber who moves around the vertices in discrete steps. There is a single cop, who is not on the graph but can probe vertices and receive information about how far away the robber is (in terms of the normal graph distance) from the vertex probed. For ease of reading we shall refer to the cop as female and the robber as male. The robber initially occupies any vertex. Each round consists of a move for the robber, in which he may move to an adjacent vertex or stay where he is, followed by a probe of any vertex by the cop. The cop then wins immediately if she is able to determine the robber's current location from the results of that probe and previous ones. This game may be viewed as a variant of the Sequential Locating game, also studied by Seager \cite{Sea13}, with the difference between the two being that in the Robber Locating game the target can move about the graph; in both games the choice of probe made may depend on the results of previous probes. If we instead require all the probes to be chosen at once (with a stationary target), we recover the Graph Locating game, independently introduced by Slater \cite{Sla75} and by Harary and Melter \cite{HM76}. 

In games with a stationary target, the searcher can guarantee to win eventually, simply by probing every vertex, and the natural question is the minimum number of probes required to guarantee victory on a given graph $G$. For the Graph Locating game, this is the \textit{metric dimension} of $G$, written $\mu(G)$. In the Robber Locating game, by contrast, it is not necessarily true that the cop can guarantee to win in any number of probes. Consequently the primary question in this setting is whether, for a given graph $G$, the cop can guarantee victory in bounded time on $G$, or equivalently whether she can catch a robber who has full knowledge of her strategy. We say that a graph is \textit{locatable} if she can do this and \textit{non-locatable} otherwise. 

In the game as introduced by Seager there was an additional rule that the robber cannot move to the vertex probed in the previous round (the \textit{no-backtrack condition}). Carragher et al.\ considered the game without this restriction, as do we, and Seager also considered the version without the no-backtrack rule for trees \cite{Sea14}. A similar game in which the searcher wins only if she probes the current location of the target and receives no information otherwise, but the target must move at each turn, was recently analysed by one of the authors \cite{Has13}. 

The main result of Carragher et al.\ \cite{CCDEW} is that for any graph $G$ a sufficiently large equal-length subdivision of $G$ is locatable. Formally, write $G^{1/m}$ for the graph obtained by replacing each edge of $G$ by a path of length $m$ through new vertices. Carragher et al.\ proved that $G^{1/m}$ is locatable whenever $m\geqslant \min\{\abs{V(G)},1+\max\{\mu(G)+2^{\mu(G)},\Delta(G)\}\}$. In most graphs this bound is simply $\abs{V(G)}$, and they conjectured that this was best possible for complete graphs, i.e.\ that $K_n^{1/m}$ is locatable if and only if $m\geqslant n$. The present authors \cite{HJK} showed that in fact $K_n^{1/m}$ is locatable if and only if $m\geqslant n/2$, for every $n\geqslant 11$. In this paper we show that the same improvement may be obtained in general: provided $\abs{V(G)}\geqslant 23$, $G^{1/m}$ is locatable whenever $m\geqslant\abs{V(G)}/2$. This bound is best possible, since $K_n^{1/m}$ is not locatable if $m=(n-1)/2$, and some lower bound on $\abs{V(G)}$ is required for it to hold, since $K_{10}^{1/5}$ is not locatable \cite{HJK}. These results, and those of Carragher et al., fundamentally depend on taking equal-length subdivisions; in the final section of this paper we show that an unequal subdivision is also locatable provided every edge is subdivided by at least a certain amount.

\section{Subdivisions and maximal matchings}

Recall that $G^{1/m}$ is the graph obtained by replacing each edge of $G$ with a path of length $m$ through new vertices. Each such path is called a \textit{thread}, and an \textit{branch vertex} in $G^{1/m}$ is a vertex that corresponds to a vertex of $G$. We write $u\thrd v$ for the thread between branch vertices $u$ and $v$.  We use ``a vertex on $u\cdots v$'' to mean any of the $m+1$ vertices of the thread, but ``a vertex inside $u\cdots v$'' excludes $u$ and $v$.

Our basic strategy to locate the robber on sufficiently large equal-length subdivisions of $G$ is to ensure the following.
\begin{enumerate}
\item Whenever the robber is at a branch vertex, the probe we make reveals that fact.
\item If the robber ever spends $r$ turns without visiting a branch vertex, we establish which thread he is inside and then can win on the next turn, where $r$ depends only on $G$.
\item If, when the robber visits a branch vertex, there is more than one possibility for his location, we can ensure that by the next time he is at a branch vertex we reduce the set of possibilities to a simpler set.
\end{enumerate}
1 and 2 above are sufficient to ensure that $G^{1/m}$ is locatable for sufficiently large $m$, since if $m>r$ the robber can only ever visit one branch vertex without being caught, and we can eventually find which one it is. For smaller $m$ the robber may be able to visit several branch vertices, so to get a better bound we need a way to progressively reduce the possibilities as in 3. This reduction is not necessarily to a smaller set but we ensure that only a bounded number of reductions can occur (and, by 2, each takes bounded time) before we reach a singleton set. Note that the reduction occurs by the next time the robber is at a branch vertex, even if this is because he stays at his current branch vertex until the next turn.

To get a bound of close to $\abs{V(G)}/2$, roughly speaking, our strategy is that in between the robber's visits to branch vertices we aim to eliminate possible destinations in pairs. To do this we must probe inside threads, aiming to eliminate both ends of the thread. This approach is simplified in the case of complete graphs, analysed extensively in \cite{HJK}, by the knowledge that any pair of branch vertices has a thread between them. In the general case we need some knowledge of the structure of $G$. To that end we will, so far as this is possible, divide the vertices of $G$ into adjacent pairs. So we take a maximal matching in $G$, and we will get a bound which depends on the size of that maximal matching.

Fix a connected graph $G$ with $n$ vertices, and let $M$ be a maximal matching. Write $k$ for $\abs{M}$ and $X$ for $V(G)\setminus V(M)$; since $M$ is maximal, $X$ is an independent set. Note that $M$ must be maximal, i.e.\ inextensible, but need not be a maximum-size matching.

\begin{thm}\label{match}If $m\geqslant k+1$ and $m\geqslant 12$, $G^{1/m}$ is locatable.
\end{thm}

The bound $m\geqslant k+1$ is best possible: $K_{2k+1}$ has a maximal matching of size $k$, but $K_{2k+1}^{1/k}$ is not locatable (\cite{HJK}, Theorem 3).

If $m$ is odd, let $t=(m-1)/2$. We use the word ``midpoint'' to refer to the vertex at distance $m/2$ from each end of a thread if $m$ is even, and either of the two vertices at distance $t$ from one end if $m$ is odd. If $m$ is odd, we also use the term ``off-midpoint'' to refer to either vertex at distance $t-1$ from one end and $t+2$ from the other. Note that probing a branch vertex will establish the robber's exact distance to his nearest branch vertex (by considering the result mod $m$). Probing a midpoint will also establish the robber's exact distance to the nearest branch vertex when $m$ is even, but when $m$ is odd it will only give this distance within two consecutive possibilities. This uncertainty makes the odd case significantly more complicated. In particular, note that point 1 above can be assured by probing branch vertices and midpoints when $m$ is even, but only by probing branch vertices when $m$ is odd. For this reason our strategy reverts to probing branch vertices whenever it is possible, given the results of previous probes, that the robber is at a branch vertex. (This is not necessary when $m$ is even, but we present a strategy which will work in either case rather than considering the two cases separately.) For the remainder of this section we assume that $m\geqslant k+1$ and $m\geqslant 12$. 

Before giving the proof of Theorem~\ref{match} we prove some preparatory lemmas. 

\begin{lem}\label{astar}Suppose that, immediately following some probe, we know that the robber is on a thread between $u$ and a vertex in $Z$, where $u$ is any specified vertex and $Z\subseteq X$. Then there is a strategy that will, within $\abs{Z}$ steps and by the next time the robber reaches a branch vertex, either win or identify that he is in $Z$.
\end{lem}
\begin{proof}We may neglect any vertices in $Z$ which do not have threads leading to $a$. Simply probe the remaining vertices of $Z$ in order. If the robber reaches a branch vertex, we will identify whether it is the one just probed (distance 0), $a$ (distance $m$), or some other vertex in $Z$ (distance a higher multiple of $m$, since no threads go between vertices in $X$). If he does not reach a branch vertex within $\abs{Z}$ turns then we will have probed one end of the thread he is on (identified by getting a result between 0 and $m$) and so located him (since the other end is known to be $a$).
\end{proof}

\begin{lem}\label{noturnback}Suppose that, immediately following some probe, we know that the robber is adjacent to a branch vertex. Write $x$ for the unknown branch vertex he is adjacent to. Let $v$ be any vertex and $w_1w_1',\ldots,w_{k-2}w_{k-2}'$ be a set of $k-2$ edges of $G$. Then it is possible to make at most $2k-3$ probes in such a way that
\begin{enumerate}[(i)]
\item if the robber is at a branch vertex for one of these probes, then the cop identifies this fact and that branch vertex must be $x$;
\item otherwise, for each of the sets $\{v\},\{w_1,w'_1\},\ldots,\{w_{k-2},w_{k-2}'\}$ in turn, the cop identifies how many endpoints of the robber's current thread are in that set.
\end{enumerate}
\end{lem}
\begin{proof}
We start by probing $v$. This establishes whether the robber is at $x$, distance 1 from $x$ or distance $2$ from $x$. If the robber is at $x$ then we are done by (i). Otherwise he is inside a thread and the first probe also establishes whether that thread meets $v$. We now consider $w_1w_1',\ldots,w_{k-2}w_{k-2}'$ in order. For each thread $w_i\thrd w_i'$, we either probe a vertex near the middle or probe both endpoints in turn. We do the former unless it is possible (given previous probe results) that the robber is at a branch vertex by the time of this probe. When probing a vertex near the middle, we use a midpoint unless $m$ is odd, $m=k+1$, $i=t$, and the result of the previous probe was consistent with the robber being at distance $t+1$ from $x$ and with him being at distance $t-1$ from $x$ at that time; in this specific case we use an off-midpoint. We refer to the probe or probes on $w_i\thrd w'_i$ as ``stage $i$''.

First we analyse the case where we probe an off-midpoint. If we do this, $m$ must be odd and the previous probe must have been at a midpoint, since if it were at a branch vertex there would be no uncertainty about the robber's distance to the nearest branch vertex. In order for the result of the previous probe to meet the conditions for us to probe an off-midpoint, it must be $\pm 1$ (mod $m$), since if it were a multiple of $m$ then the robber being $t-1$ steps from a branch vertex would be impossible and if it were anything else then the robber being $t+1$ steps from a branch vertex would be impossible. Consequently, at the time of the previous probe the robber was at a midpoint or one step away from a midpoint. We now probe the off-midpoint of $w_t\thrd w'_t$, and at this point the robber can be at most two steps from a midpoint. So if he is on $w_t\thrd w'_t$ the distance returned is at most 4, if he is on an adjacent thread it is in $\{2t-3,\ldots,2t+5\}$, i.e.\ in $\{m-4,\ldots,m+4\}$, and if he is anywhere else it is at least $2m-4$. Since $m$ is odd we must have $m\geqslant 13$, and we will successfully distinguish these possibilities.

Likewise, if we probe a midpoint of a thread and get a result $r$ then, since we know the robber cannot be at a branch vertex, $r\leqslant t$ if he is inside the probed thread, $t+1\leqslant r \leqslant t+m$ if he is inside an adjacent thread, and $r>t+m$ otherwise. Therefore, by this method we establish in stage $i$ whether the robber has reached a branch vertex and, if not, whether he is on a thread ending in $w_i$ or $w_i'$ (and whether he is on $w_i\thrd w'_i$). It remains to prove that if the robber is at a branch vertex for one of these probes then that vertex is $x$.

We claim that at each probe in stage $i$ the robber's distance from $x$ is at most $i+2$, and that equality is possible for stage $i$ only if $i\leqslant 2$ or we had equality for the final probe in stage $i-1$; this is sufficient since $i+2\leqslant k<m$. If $i=1$ then the previous probe was at $v$ and established the robber's exact distance from $x$; if this was 2 then we need only one additional probe for $w_1\thrd w'_1$ and if it was 1 we need two additional probes. So the claim is true for $i=1$. For $i\in\{2,\ldots k-3\}$ we proceed by induction. The robber's distance from $x$ was at most $i+1$ at the time of the previous probe. If the previous probe was consistent with him being at distance 1 from a branch vertex, then he must have been distance at most 2 from a branch vertex; since $i+1\leqslant k-2<m-2$, his distance from $x$ was at most 2, and so for each of the two probes required for this stage it is at most $4\leqslant i+2$; if $i>2$ the inequality is strict. If the previous probe was not consistent with him being next to a branch vertex, we only require one probe for this stage so the robber's distance from $x$ when we make this probe is at most $i+2$, and it is less unless there was equality for the previous probe. 

Finally, for $i=k-2$, we again know that the robber's distance from $x$ was at most $k-1$ at the time of the previous probe. Again, if the result of the previous probe was consistent with him being next to a branch vertex he must in fact have been within two steps from a branch vertex (and, if exactly two steps away, $m$ must be odd). Suppose he was two steps from a branch vertex other than $x$, i.e.\ at distance $k-1$ from $x$. In this case, the previous probes must be consistent with him being at distance $i+2$ at the end of stage $i$ for every $i$. If the result of stage $t-1$ was inconsistent with him being $t-1$ away from $x$ then (since it was consistent with him being $t+1$ away) the robber was at least $t$ away from $x$ at that point, and we therefore know he is not adjacent to a branch vertex at the end of stage $k-3$. Conversely, if the result of stage $t-1$ was consistent with the robber being $t-1$ from $x$ at the end of that stage, then we probed an off-midpoint at stage $t$. When we probed the off-midpoint, the robber must have been at distance $t+2$ from $x$ and so the result of that probe (mod $m$) would be in $\{0,\pm 3\}$. If he had been at distance $t-2$ from $x$, the result would have been in $\{\pm 1,\pm 4\}$, so we can eliminate this possibility. Consequently, if the robber is at distance $k-1$ at the end of stage $k-3$ then we will know that he is not adjacent to a branch vertex, and so we only use one probe for stage $k-2$ and his distance from $x$ is at most $k$ when we make it, as required. Otherwise, even if we use two probes for stage $k-2$, his distance at each of them is at most $k$. This proves the claim. 
\end{proof}

\begin{lem}\label{gamma}Suppose that, immediately following some probe, we know that the robber is at the vertex at distance 1 from $Z$ on a thread between a vertex in $V(M)\setminus\{a\}$ and a vertex in $Z$, where $Z\subseteq X$. Then there is a strategy that will, within $2k-2+\abs{Z}$ steps and by the next time the robber reaches a branch vertex, either win or identify that he is in $Z$.
\end{lem}
\begin{proof}
If $k=1$ then $V(M)\setminus \{a\}=\{a'\}$, and so the result is true by Lemma~\ref{astar}. Otherwise, choose any $bb'\in M$ not equal to $aa'$, and take $v=a'$ and $w_1w_1',\ldots,$ $w_{k-2}w_{k-2}'$ to be the edges in $M$ other than $aa'$ and $bb'$. Now we probe as in Lemma~\ref{noturnback}, stopping if we identify that the robber is at a branch vertex or if any probe indicates that he is on an adjacent thread. By Lemma~\ref{noturnback}, this takes at most $2k-3$ probes, and if he reaches a branch vertex we have identified that he is in $Z$. If none of the probes indicate that he is in $Z$ or on an adjacent thread to the one probed then he must be on a thread adjacent to $b\thrd b'$ (since one end of his thread is outside $X$).

Consequently, within $2k-3$ probes, we have identified, at the time of the last probe, either that the robber was in $Z$ or that he was inside a thread leading from $Z$ to either $c$ or $c'$ for some $cc'\in M$ (perhaps only one of these is actually possible). In the latter case the result of the last probe either told us that the robber was distance at most 2 from the nearest branch vertex or that he was distance at least 2 from the nearest branch vertex. If we know he was not adjacent to a branch vertex then we may probe $c$ to establish whether he is inside a thread leading to $c$ or to $c'$, and then we are done in at most $\abs{Z}$ additional steps by Lemma~\ref{astar}. If not then probe the vertex at distance 4 from $c$ along $c\thrd c'$. If the robber is now in $Z$, this will return $m+4$ or $2m-4$. If he is at $c$ or $c'$ then it will return 4 or $m-4$ respectively. If he is inside a thread leading to $c$ then it will return a result in $\{5,6,7,m+1,m+2,m+3\}$ and if he is inside a thread leading to $c'$ then it will return a result in $\{m-3,m-2,m-1,m+5,m+6,m+7,2m-7,2m-6,2m-5\}$. Since $m\geqslant 12$ these five possibilities are all distinguished, so either we are done immediately or within $\abs{Z}$ additional steps by Lemma~\ref{astar}.
\end{proof}

\begin{lem}\label{delta}Suppose that, immediately following some probe, we know that the robber is at the vertex at distance 1 from $a$ on a thread not leading to $a'$, where $aa'\in M$. Then there is a strategy that will, within $2k-2+\abs{X}$ steps and by the next time the robber reaches a branch vertex, either win or identify that he is in $X$.
\end{lem}

\begin{proof}
The strategy here proceeds in much the same way as for Lemma~\ref{gamma}. If $k=1$ then the robber is on a thread between $a$ and $X$ and we are done by Lemma~\ref{astar}. Otherwise, choose any $bb'\in M$ not equal to $aa'$, take $v=b$ and $w_1w_1',\ldots,w_{k-2}w_{k-2}'$ to be the edges in $M$ other than $aa'$ and $bb'$, and proceed as in Lemma~\ref{noturnback}. 

If a probe reveals that the robber is on an adjacent thread, interrupt the process. If we know which thread he is on, we can win by probing either end; if not we know he is on $a\thrd w_i$ or $a\thrd w'_i$, and from the result of the last probe we also know either that he not adjacent to $a$ or that he is not adjacent to the other end. Consequently, probing $w_i$ will either win or identify that he is inside $a\thrd w'_i$, and in the latter case we can now win by probing either end.

If this does not happen, and we do not establish that the robber has returned to $a$, then at the time of the last probe the robber was inside a thread leading to $b'$ or to some vertex in $X$. If there is no thread $a\thrd b'$ then we are done by Lemma~\ref{astar}, so we assume that there is. We know from the result of the last probe either that his distance to the nearest branch vertex was at least 2 or that it was at most 2 (and at most 1 if $m$ is even). In the former case we probe $b'$; this wins if he is on $a\thrd b'$ and we are done by Lemma~\ref{astar} if not. In the latter case we probe the vertex at distance 3 from $a$ along the thread $a\thrd b'$. If the robber is at $a$ the result will be 3, if he is at $b'$ it will be $m-3$, and if he is in $X$ it will be $m+3$. If he is inside $a\thrd b'$, the result will be in $\{0,1,2,m-6,m-5,m-4\}$ ($\{1,2,m-5,m-4\}$ if $m$ is even) and if he is inside a thread between $a$ and $X$ it will be in $\{4,5,6,m,m+1,m+2\}$ ($\{4,5,m+1,m+2\}$ if $m$ is even). Since $m\geqslant 12$ these possibilities are distinguished, and we have either won, established that he is in $X$, or are done by Lemma~\ref{astar}. 
\end{proof}

We are now ready to combine these elements into a complete strategy to locate the robber.

\begin{proof}[Proof of Theorem~\ref{match}]
We describe a winning strategy for the cop. First, probe branch vertices in turn until the answer reveals that the robber is at a branch vertex. Either this eventually happens or he remains inside a single thread, in which case we eventually identify both ends and win.

Suppose that we know, on receiving the distance from some probe, that the robber's current location is in a set $A\cup Y$ of branch vertices, where $A\subseteq V(M)$ and $Y\subseteq X$. We show that, by the time the robber next reaches a branch vertex and within time $n+2$,
\begin{enumerate}[(i)]
\item\label{inx} if $A=\varnothing$, we either locate him or reach a point where he is known to be in a smaller set $Y'\subset Y$;
\item\label{pair} if $A=\{a\}$ or $A=\{a,a'\}$, where $aa'\in M$, we either locate him or reduce to \eqref{inx};
\item\label{rest} otherwise we either locate him, reduce to \eqref{inx} or \eqref{pair}, or reach a point where he is known to be in a smaller set $A'\cup Y$ with $A'\subset A$.
\end{enumerate}
Note that when we reduce to an earlier case the number of candidate vertices in $X$ may increase; going from knowing he is in $\{a,a'\}$ to knowing he is in $X$ is a valid reduction.

In case \eqref{inx}, pick a vertex $y\in Y$ and probe a neighbour of $y$, i.e.\ a vertex on $y\thrd a$ for some $a\in V(M)$. If the response is 0 or 1 the robber is caught. If it is 2 we know the robber is one step away from $y$ heading for some vertex in $V(M)\setminus\{a\}$. Setting $Z=\{y\}$ in Lemma~\ref{gamma}, we can win in time $2k-1$ (and by the time he next reaches a branch vertex). If the response is larger, we will take $Y'=Y\setminus\{y\}$. If the response is $2m-2$ we know the robber is one step away from some vertex in $Y'$ on a thread leading to $a$, and so we can win or reach a position where the robber is known to be in $Y'$ in the required time by Lemma~\ref{astar}. If the response is $\pm 1$ mod $m$ (but not 1), the robber is known to be in $Y'$ and we are done. The only other
possibility is that the answer is 0 or $\pm 2$ mod $m$, but larger than $2m-2$, in which case the robber is known to be one step from a vertex in $Y'$, on a thread which does not lead to $a$. Now, by Lemma~\ref{gamma}, we can win or reach a point where he is known to be in $Y'$ within the required time.

In case \eqref{pair}, probe the vertex at distance 2 from $a$ on $a\thrd a'$. If the response is 1, 2, $m-3$ or $m-2$, we have won. If it is 3 then we know the robber is one step from $a$ on a thread not leading to $a'$, and if it is $m-1$ then we know he is one step from $a'$ on a thread not leading to $a$; in either of these cases we are done by Lemma~\ref{delta}. If the response is $m+1$ then the robber is on a thread between $Y$ and $a$, and we are done by Lemma~\ref{astar}. If it is greater than $m+1$ and $\pm 2$ mod $m$, the robber is in $Y$ and so we are done.  Finally, if the response is greater than $m+1$ and not $\pm 2$ mod $m$ then the robber must be one step from a vertex in $Y$ on a thread not leading to $a$, so we are done by Lemma~\ref{gamma}.

In case \eqref{rest}, first choose some $aa'\in M$ such that $a\in A$ and probe $a$. If the result is 0 we have won and if it is any higher multiple of $m$ we know the robber is in $Y\cup A\setminus\{a\}$, so are done. Otherwise we either know that the robber is adjacent to $a$ or that he is adjacent to a vertex in $Y\cup A\setminus\{a\}$. Now choose any $bb'\in M$ not equal to $aa'$, take $v=a'$ and $w_1\thrd w_1',\ldots,w_{k-2}\thrd w_{k-2}'$ to be the threads corresponding to edges in $M$ other than $aa'$ and $bb'$, and proceed as in Lemma~\ref{noturnback}. If the robber returns to a branch vertex, either we know it is $a$ and have won, or we know it is in $Y\cup A\setminus\{a\}$ and are done. If he does not, we establish whether the robber is on a thread containing $a$ or $a'$, and for each $w_i\thrd w'_i$ in turn we establish whether he is on that thread, on an adjacent thread, or elsewhere. 

If one of these probes locates the robber's thread exactly then we interrupt the process and probe either end of that thread to win. If we establish that one end of the robber's thread is $x$ and the other is $w_i$ or $w'_i$, for some $x$ and $i$, then, since the previous probe was either at $x$ or a midpoint or off-midpoint of $w_i\thrd w'_i$, the result of that probe can be consistent with the robber being adjacent to $x$ or with him being adjacent to the other end of his thread, but not both. So we interrupt the process and probe $w_i$. If the result is $m$ we know whether the robber is at $x$ or $w'_i$, so have won; if it is less than $m$ we have located him on $x\thrd w_i$ and if it is more than $m$ he is on $x\thrd w'_i$ and we can locate him by probing $x$.

If we establish that one end of the robber's thread is $w_i$ or $w'_i$ and the other is $w_j$ or $w'_j$, for some $i<j$, then again we interrupt the process. Since the last probe was at a midpoint or off-midpoint of $w_j\thrd w'_j$, the result of that probe can be consistent with the robber being within distance 2 of $\{w_i, w'_i\}$ or with him being within distance 2 of $\{w_j,w'_j\}$ but not both. Assume without loss of generality that we know he is not that close to $\{w_j,w'_j\}$. Now probe $w_i$. If the result is 0 or $m$ the robber is at $w_i$ or $w'_i$ respectively and we win; otherwise we have established which of these vertices is an end-point of his thread, and can win by probing $w_j$ and then (if necessary) $w'_j$.

If none of these occur, we continue to the end of the $k-2$ stages of Lemma~\ref{noturnback}. The final probe will either establish that the robber is distance at most 2 from a branch vertex or establish that he is distance at least 2 from every branch vertex. Further, we will have established one of the following.
\begin{enumerate}[(a)]
\item Both ends of the robber's thread are in $\{b,b'\}\cup X$.
\item One end of the robber's thread is in $\{c,c'\}$ and the other is in $\{b,b'\}\cup X$, for some $cc'\in M$ (this includes the case where we further know which of $c$ and $c'$ it is; we will not use this information).
\end{enumerate}
In case (a), by probing vertices in $\{b,b'\}\cup X$ in turn, we will either win or establish that he has reached a branch vertex in $\{b,b'\}\cup X$, reducing to \eqref{pair} as desired.

We therefore assume case (b). If the robber was known not to be adjacent to a branch vertex at the last probe, then probe a midpoint of $b\thrd b'$. This will reveal if the robber is on a thread ending in $b$ or $b'$. If so, it will also establish either that he is distance at least 2 from $\{b,b'\}$ or that he is distance at least 2 from $\{c,c'\}$, without loss of generality the former. We can then win by probing $c$ followed by $b$ and $b'$. If the probe reveals that the robber is not in a thread meeting $b$ or $b'$, both ends of his thread are in $\{c,c'\}\cup X$, and by probing vertices in this set in turn we will win or reduce to \eqref{pair}. 

Alternatively, if the robber was known to be within 2 of a branch vertex at the end of the $k-2$ stages of Lemma~\ref{noturnback}, probe the vertex at distance 3 from $b$ along $b\thrd b'$. If the result is 3 or $m-3$ the robber is at $b$ or $b'$ respectively and is caught. If the result is larger and $\pm 3$ mod $m$, he is in $\{c,c'\}\cup X$. If the result is 4, 5 or 6 he is on $c\thrd b$ or $c'\thrd b$ near $b$, if it is $m+1$ or $m+2$ he is on one of these threads near the other end, and if it is $m-2$ or $m-1$ he is on $c\thrd b'$ or $c'\thrd b'$ near $b'$; in any of these cases probing $w_i$ will locate him. If the result is $m$ then the robber may be on $c\thrd b'$ or $c'\thrd b'$ near $b'$ or on $c\thrd b$ or $c'\thrd b$ near the other end, but he is not adjacent to a branch vertex so probing $b$ will determine which of these two possibilities is the case and then probing $c$ will locate him. If none of these apply then he is not on a thread meeting $b$ or close to $b'$, so probing $b'$ will determine whether he is on $c\thrd b'$ or $c'\thrd b'$. If he is we can locate him as before; if not both ends of his thread are in $\{c,c'\}\cup X$ and we can probe vertices in this set until we locate him or reduce to \eqref{pair}.

This completes the analysis of cases, and so from any point where we know that the robber is in a set of branch vertices, we can reduce that set using \eqref{inx}, \eqref{pair} or \eqref{rest}. It is simple to check that in each case at most $2k+2+\abs{X}=n+2$ probes are required. At most $n-1$ reductions can occur before the set is a singleton and the robber is located, and so this strategy guarantees to catch the robber in bounded time.
\end{proof}

\section{Imperfect maximal matchings}

Theorem~\ref{match} gives a bound in terms of the size of a maximal matching, and where there are maximal matchings of different sizes we are free to choose that which gives the best bound, i.e.\ the smallest one. This bound is therefore weakest when all maximal matchings are perfect matchings. We next show that, since $G$ is connected, there are only two possibilities for such a graph. Write $\mmm(G)$ for the minimum size of a maximum matching of $G$.

\begin{lem}\label{mmm}If $G$ is a connected graph with $2r$ vertices such that $\mmm(G)=r$, then either $G\cong K_{2r}$ or $G\cong K_{r,r}$.
\end{lem}

\begin{proof}If $r=1$ then this is trivial. If $r=2$ and $G$ contains a vertex of degree 3, the each of those three edges can be extended to a matching of size 2, so every other pair of vertices is adjacent and $G\cong K_4$. If $r=2$ and $G$ has no vertex of degree 3 then, since $G$ is connected, $G\cong P_4$ or $G\cong K_{2,2}$, but $\mmm(P_4)=1$ by taking the middle edge. So the result is true for $r=2$. 

Let $G$ be such a graph for some $r>2$, and suppose the result is true for $r-1$. First note that if $G$ had a cutvertex, $v$, then $G-v$ would have at least one odd component, $C_1$ say, and another component $C_2$; letting $w$ be a neighbour of $v$ in $C_2$, $G-\{v,w\}$ would have an odd component, so the largest matching containing $vw$ would have size less than $r$, contradicting $\mmm(G)=r$. So $G$ is 2-connected. 

\begin{clm}If $vw$ is any edge of $G$ then $G-\{v,w\}$ is connected.
\end{clm}
\begin{poc}Suppose not. If any component of $G-\{v,w\}$ is odd, then $vw$ cannot be extended to a matching of size $r$, contradicting $\mmm(G)=r$. If all components are even then let $C_1$ and $C_2$ be any two components. Since $G$ is 2-connected, $v$ and $w$ each have neighbours in every component. Let $x$ be a neighbour of $v$ in $C_1$ and $y$ be a neighbour of $w$ in $C_2$. Now $C_1-x$ has an odd number of vertices, so at least one odd component, which is also an odd component of $G-\{v,w,x,y\}$. Consequently $\{vx,wy\}$ cannot be extended to a matching of size $r$, contradicting $\mmm(G)=r$. So the claim is proved.
\end{poc}

Since $G-\{v,w\}$ is connected for every $vw\in E(G)$, by the induction hypothesis it is isomorphic to $K_{2(r-1)}$ or $K_{r-1,r-1}$, since otherwise $G$ has a maximal matching of size $1+\mmm(G-\{v,w\})<r$. Next we show that the graph so obtained is isomorphic to the same one of these for every edge.

\begin{clm}Either $G-\{v,w\}\cong K_{2(r-1)}$ for every $vw\in E(G)$, or $G-\{v,w\}\cong K_{r-1,r-1}$ for every $vw\in E(G)$.
\end{clm}
\begin{poc}Suppose $uv,vw\in E(G)$. Since $r>2$, there are vertices $x,y,z\not\in\{u,v,w\}$. If these vertices form a triangle in $G$ then $G-\{u,v\}\not\cong K_{r-1,r-1}$, so $G-\{u,v\}\cong K_{2(r-1)}$, and similarly $G-\{v,w\}\cong K_{2(r-1)}$. If they do not then neither graph is complete, so each must be isomorphic to $K_{r-1,r-1}$. So any pair of edges which share a vertex produce the same graph, and since $G$ is connected the same is true for any pair of edges.
\end{poc}
If $G$ is not complete then it has non-adjacent vertices $x$ and $y$. If every edge meets $x$ or $y$ then $\mmm(G)\leqslant 2<r$. So some edge does not, and then removing the endpoints of that edge gives a non-complete graph. So either $G$ is complete or $G-\{v,w\}\cong K_{r-1,r-1}$ for every $vw\in E(G)$.

Suppose $G-\{v,w\}\cong K_{r-1,r-1}$ for every $vw\in E(G)$. Take disjoint edges $uv,wx,yz$ (this is possible since $G-\{u,v\}$ has $r-1$ disjoint edges). It is not possible for two vertices to appear in the same part of one of the graphs $G-\{u,v\},G-\{w,x\},G-\{y,z\}$ and different parts of another, so we can 2-colour the vertices of $G$ consistently with each of the colourings of these three graphs. Now any two vertices both appear in at least one of the three graphs, so are adjacent if and only if they are opposite colours. Since this means that $u$ and $v$ are opposite colours, and $G-\{u,v\}\cong K_{r-1,r-1}$, there are equal numbers overall and $G\cong K_{r,r}$.
\end{proof}

This allows us to complete the proof of our main theorem.

\begin{thm}Let $G$ be a connected graph of order $n$, where $n\geqslant 23$. Then $G^{1/m}$ is locatable for any $m\geqslant n/2$.
\end{thm}
\begin{proof}If $m\geqslant n/2$ then also $m\geqslant 12$, since $n\geqslant 23$. If $G=K_n$ then $G^{1/m}$ is locatable by Theorem 4 of \cite{HJK}. If $G=K_{n/2,n/2}$ then $G^{1/m}$ is locatable by Theorem 5 of \cite{HJK}. Otherwise, by Lemma~\ref{mmm}, $G$ has a maximal matching $M$ of size $k$, where $k<n/2\leqslant m$. Since $k$ and $m$ are integers, $m\geqslant k+1$ and so, by Theorem~\ref{match}, $G^{1/m}$ is locatable.
\end{proof}

\section{Unequal subdivisions}

So far all our results, like those of \cite{CCDEW}, have assumed equal-length subdivisions. Carraher et al.\ conjectured that this restriction was not necessary, and that further subdividing a locatable graph always gives another locatable graph, but the present authors \cite{HJK} and Seager \cite{Sea14} independently gave a counterexample to this conjecture. Here we present a proof that a subdivision where the edges of $G$ are replaced by paths of arbitrary length is locatable, provided that the minimum length of path used is sufficiently large.

For a given graph $G$ and function $l:E(G)\to\mathbb{Z}^+$, define $G^{1/l}$ to be the graph obtained from $G$ by replacing each edge $e\in E(G)$ by a path of length $l(e)$. 
\begin{thm}For any graph $G$ with $n$ vertices, if $l(e)\geqslant 2n$ for every $e\in E(G)$ then $G^{1/l}$ is locatable.
\end{thm}
\begin{proof}We give a winning strategy on $G^{1/l}$. First probe each branch vertex in turn, and write $d_x$ for the distance obtained from probing $x$. Let $(u,v)$ be a pair for which $(d_u,d_v)$ is as small as possible (lexicographically). Now consider the robber's possible positions at the time of the next probe. 

Write $x$ for the nearest branch vertex and $d$ for the robber's distance from that branch vertex. If the robber is not on a thread containing $u$ then $d_u\geqslant 2n+d-n$ and $d_x\leqslant d+n$. We cannot have equality in both cases, since the first would require $u$ to have been the first vertex probed, and the second would require $x$ to have been the first vertex probed. Thus $d_u>d_x$, contradicting the definition of $u$. So the robber must be on some thread containing $u$; let the other end of the thread be $w$. The shortest path from $u$ to the robber's current location is along $u\thrd w$, for if not $d_u\geqslant 2n+d-n$ and $d_w\leqslant d+n$ (and again equality cannot occur in both cases simultaneously).

Suppose $x\neq u$, i.e. the robber is nearer to $w$ than to $u$. Then $d_w\geqslant d+n$, but if $v\neq w$ then $d_v\leq 2n+d-n$. Again, we can have equality in at most one of these, so $v\neq w$ would imply $d_v>d_w$, contradicting the definition of $v$. Consequently either $x=u$ or $x=v$.

Now probe $v$. If the distance returned is at most $n$, the robber must be on $u\thrd v$ and we have won. Otherwise, probe $u$, and write $d'$ for the distance returned. If $d'=0$ we have won, so assume $d'>0$. If $w=v$ then we know from the previous probe that the robber's distance from $w$ (at the time of probing $u$) is at least $n-1$, whereas if $w\neq v$ then we know the robber's distance from $w$ (at the time of probing $u$) is at least $l(uw)/2-1$, which is again at least $n-1$ (since otherwise the robber was nearer to $w$ than $u$ on the previous turn, so $w=v$). For every branch vertex $y$ with $uy\in E(G)$ and $d'+n-1\leqslant l(uy)$, select the $(d'+n-1)$th vertex on $u\thrd y$ (counting $u$ as the $0$th). Note that every possible candidate for $w$ is covered, and there are at most $n-1$ vertices selected. Probe these vertices in turn until the distance returned is at most $\min(2(n-1),d'+n-1)$. This happens if and only if the robber is on the thread being probed (including the case where the robber has returned to $u$), so it must eventually occur. When it does we have identified the robber's location uniquely.
\end{proof}

\section{Acknowledgements}
The first author acknowledges support from the European Union through funding under FP7--ICT--2011--8 project HIERATIC (316705), and is grateful to Douglas B. West for drawing his attention to this problem. The second author acknowledges support through funding from NSF grant DMS~1301614 and MULTIPLEX grant no.\ 317532, and is grateful to the organisers of the 8th Graduate Student Combinatorics Conference at the University of Illinois at Urbana-Champaign for drawing his attention to the problem. The third author acknowledges support through funding from the European Union under grant EP/J500380/1 as well as from the Studienstiftung des Deutschen Volkes.


\begin{thebibliography}{99}

\bibitem{AF84} M. Aigner and M. Fromme, 
A game of cops and robbers, 
\textit{Discrete Appl. Math.} \textbf{8} (1984), 1--11.

\bibitem{BLK13} B. Bollob{\'a}s, G. Kun and I. Leader,
Cops and robbers in a random graph,
\textit{J. Comb. Theory, Ser. B} \textbf{103} (2013), 226 -- 236.

\bibitem{CCDEW} J. Carraher, I. Choi, M. Delcourt, L.H. Erickson and D.B. West, 
Locating a robber on a graph via distance queries, 
\textit{Theoretical Comp. Sci.} \textbf{463} (2012), 54--61.

\bibitem{CN00} N. E. Clarke and R. J. Nowakowski,
Cops, robber, and photo radar,
\textit{Ars Combin.} \textbf{56} (2000), 97--103. 

\bibitem{Fra87} P. Frankl, 
Cops and robbers in graphs with large girth and Cayley graphs,
\textit{Discrete Appl. Math.} \textbf{17} (1987), 301--305.

\bibitem{HM76} F. Harary and R. A. Melter,
On the metric dimension of a graph,
\textit{Ars Combin.} \textbf{2} (1976), 191--195.

\bibitem{Has13} J. Haslegrave, 
An evasion game on a graph,
\textit{Discrete Math.} \textbf{314} (2014), 1--5.

\bibitem{HJK} J. Haslegrave, R.A.B. Johnson and S. Koch,
The Robber Locating game,
\textit{Discrete Math.} \textbf{339} (2016), 109--117.

\bibitem{LP10} T. \L uczak and P. Pra\l at,
Chasing robbers on random graphs: Zigzag theorem,
\textit{Random Structures \& Algorithms} \textbf{37} (2010) 516 -- 524.

\bibitem{NW83} R. Nowakowski and P. Winkler,
Vertex to vertex pursuit in a graph, 
\textit{Discrete Math.} \textbf{43} (1983), 235--239.

\bibitem{Par76} T.D. Parsons, 
Pursuit-evasion in a graph, 
\textit{Theory and Applications of Graphs}, in \textit{Lecture Notes in Mathematics}, Springer-Verlag (1976), 426--441.

\bibitem{Qui78} A. Quillot, 
\textit{Jeux et pointes fixes sur les graphes.}
Th\`ese de 3\`eme cycle, Universit\'e de Paris VI, 1978, 131--145.

\bibitem{Sea12} S. Seager, 
Locating a robber on a graph,
\textit{Discrete Math.} \textbf{312} (2012), 3265--3269.

\bibitem{Sea13} S. Seager, 
A sequential locating game on graphs,
\textit{Ars Combin.} \textbf{110} (2013), 45--54.

\bibitem{Sea14} S. Seager,
Locating a backtracking robber on a tree,
\textit{Theoretical Computer Science} \textbf{539} (2014), 28--37.

\bibitem{Sla75} P. J. Slater, 
Leaves of trees,
\textit{Proc. Sixth Southeastern Conf. Combin., Graph Theory, Computing} in Congressus Numer. \textbf{14} (1975), 549--559.

\end{thebibliography}
\end{document}